 \documentclass[a4paper, reqno,11pt]{amsart}
\usepackage{lscape}
 \usepackage[ansinew]{inputenc}
 \usepackage[dvips]{graphicx}
\usepackage[psamsfonts]{amssymb} \usepackage[all]{xy}  \usepackage[all]{xy}
 \usepackage[psamsfonts]{eucal}
 \usepackage[psamsfonts]{eucal}
\usepackage{amssymb, amsmath}
\usepackage{amsthm}
\usepackage{geometry}
\usepackage{enumerate}
\usepackage{float}
 \usepackage{color}

 \title[Realization and crossed modules]{Realization of Lie algebras and classifying spaces of crossed modules}

\author{Yves F\'elix}
\address{Institut de Math\'ematiques et Physique\\
         Universit\'e Catholique de Louvain-la-Neuve\\
         Louvain-la-Neuve\\
         Belgique}
\email{Yves.Felix@uclouvain.be}

\author{Daniel Tanr\'e}
\address{D\'epartement de Math{\'e}matiques, 
         UMR-CNRS 8524 \\
         Universit\'e de Lille,
         F-59000 Lille}
\email{Daniel.Tanre@univ-lille.fr}

\thanks{The authors are partially supported by the 
MINECO-FEDER grant  MTM2016-78647-P.
The second author is partially supported by the  ANR-11-LABX-0007-01  ``CEMPI''}



\newcommand{\bd}{\begin{displaymath}}
\newcommand{\ed}{\end{displaymath}}
\newcommand{\ben}{\begin{enumerate}}
\newcommand{\een}{\end{enumerate}}
\newtheorem{theorem}{Theorem}[section]
\newtheorem{theoremv}{Theorem}

 \newtheorem{proposition}[theorem]{Proposition}
 \theoremstyle{definition}
 \newtheorem{definition}[theorem]{Definition}
 \theoremstyle{remark}
 \newtheorem{remark}[theorem]{Remark}
 
 \numberwithin{equation}{section}
\newcommand{\lasu}{\mathfrak{L}}
\def\L{\mathbb{L}}
\def\hL{\widehat{\mathbb L}}
\def\ov{\overline}
\def\ner{{\mathrm{Ner}}}

\def\cdgl{{\mathbf{cdgl}}}
\def\sset{{\mathbf{Sset}}}
\def\crmod{{\mathbf{CrMod}}}
\def\cC{{\mathcal{C}}}

\def\Hom{{\mathrm{Hom}}}
\def\Der{{\mathrm{Der}}}
\def\Aut{{\mathrm{Aut}}}
\def\ad{{\mathrm{ad}}}
\def\gn{{\mathfrak{n}}}
\def\gm{{\mathfrak{m}}}

\def\gu{{\mathfrak{u}}}
\def\gv{{\mathfrak{v}}}
\def\Q{{\mathbb{Q}}}
\date{\today}
 
\begin{document}

 \begin{abstract}
The category of complete differential graded Lie algebras provides
nice algebraic models for the rational homotopy types of non-simply connected spaces. 
In particular, there is a realization functor, $\langle -\rangle$, of any complete differential graded Lie algebra as a simplicial set. 
In a previous article, we considered the particular case of a complete  graded
Lie algebra, $L_{0}$,   concentrated in degree 0 and proved  that $\langle L_{0}\rangle$   is isomorphic 
to the usual  bar construction on the Malcev group associated to $L_{0}$.

Here we consider the case of a complete differential graded
Lie algebra, $L=L_{0}\oplus L_{1}$, concentrated in degrees 0 and 1. 
We establish that the category of such two-stage Lie algebras 
is equivalent to explicit subcategories of  crossed modules and
Lie algebra crossed modules, extending the equivalence between pronilpotent Lie algebras and Malcev groups. 
In particular, there is a  crossed module
$\mathcal{C}(L)$ associated to $L$. We prove that $\mathcal{C}(L)$ 
is isomorphic to the Whitehead crossed module associated
to the simplicial pair  $(\langle L\rangle, \langle L_{0}\rangle)$.
Our main result is the identification of $\langle L\rangle$ with the classifying space of $\mathcal{C}(L)$.
 \end{abstract}
 
 \maketitle

\section*{Introduction}

 In this text, we pursue the study of the rational homotopy type of spaces with  models in the category $\cdgl$ of
 complete differential graded Lie algebras,
 as  developed in  \cite{book}. 
 We emphasize that in this approach, there are no requirements concerning simply connectivity  
 or  nilpotency. In particular, to any finite simplicial complex is associated a cdgl $M_X$ whose homology in degree $0$    is the  Malcev completion of $\pi_{1}(X)$ (\cite[Theorem 10.5]{book}).
 
 \smallskip
 One of the main tools in this theory  is a cosimplicial cdgl
 $\lasu_{\bullet}=\{\lasu_{n}\}_{n\geq 0}$, where $\lasu_0$ is the free Lie algebra on a Maurer Cartan element in degree -1, and $\lasu_1$ is the Lawrence-Sullivan interval (see below for more details). 
  This cosimplicial cdgl plays a role similar to the simplicial algebra of PL-forms on $\underline{\Delta}^{\bullet}$.
  It enables us to construct   a  realization functor 
 from the category of complete differential graded Lie algebras to the category of simplicial sets,
 $\langle -\rangle\colon \cdgl \to \sset$, defined by
 $$\langle L\rangle_{\bullet} := \Hom_{{\cdgl}} (\mathfrak{L}_{\bullet}, L).$$
 
 \medskip
 If a Lie algebra $L$ is concentrated in degree 0, we proved in \cite[Theorem 0.1]{bar} that 
 its realization  $\langle L\rangle$ is \emph{isomorphic} to the usual bar construction 
 on the group $\mathrm{exp}\, L$, 
 constructed on the set $L$  with the Baker-Campbell-Hausdorff product.

 \medskip
 Here we consider the next step: {$L$ is a connected cdgl with non-trivial homology only in degrees 0 and 1.}
 Geometrically, this  corresponds to the notion of homotopy 2-types  and, by analogy, 
 a  connected cdgl $L$ such that $H_{*}L=H_{0}L\oplus H_{1}L$ is called a \emph{2-type cdgl.}
  First of all, if $L=L_{\geq 0} $ and $H_{\geq 2} L=0$, then the Lie subalgebra 
$I=L_{\geq 2}\oplus dL_{2}$ is an ideal  because 
 if $a\in L_0$ and $b\in L_2$, then $da= 0$ and $[a, d(b)]= d[a,b]$.
 Moreover  $I$ is acyclic, and the quotient map is a quasi-isomorphism,
 $$\varphi \colon (L,d)
  \xrightarrow{\simeq} 
   (L/I, \overline{d}).$$
Therefore, since the realization functor $\langle-\rangle$ preserves 
quasi-isomorphisms of connected cdgl's (\cite[Corollary 8.2 and Remark 8.6]{book}), we get a 
weak homotopy equivalence
 $$\langle\varphi\rangle \colon \langle L,d\rangle 
 \xrightarrow{\simeq} 
 \langle L/I, \overline{d}\rangle.$$ 
 
 We have thus reduced the problem to considering only cdgl's $L$ of the form $L= L_0\oplus L_1$ and
 denote by $\cdgl_{\leq 1}$ the corresponding subcategory of $\cdgl$.
 We associate to such $L$ a natural crossed module $\cC(L)$ and denote by $\crmod$ the category of crossed modules.
Our main result, which extends \cite[Theorem 0.1]{bar},  can be formulated  as follows.
 
 \begin{theoremv}\label{thm:realizationandcrossed}
 If $L$ is a complete differential graded Lie algebra such that $L=L_{0}\oplus L_{1}$,
 then its geometric realization $\langle L\rangle$ is naturally isomorphic to the classifying simplicial set $B\cC(L)$;
 i.e., the diagram
$$\xymatrix{
\cdgl_{\leq 1}\ar[rr]^-{\langle -\rangle}\ar[d]_{\cC}&&
\sset\\
\crmod
\ar[rru]_-{B}
}$$
commutes up to natural isomorphisms.
 \end{theoremv}

  This theorem shows that  the functor $\langle -\rangle$ 
  generalizes many classical constructions.
  
  \medskip
   Geometrically,  crossed modules appear in  the work of Whitehead (\cite{MR30760}).
   If $(X,A)$ is a pair of topological spaces, based in $A$, Whitehead proved that the boundary map
   $d\colon \pi_{2}(X,A)\to \pi_{1}(A)$, together with the action of $\pi_{1}(A)$ on $\pi_{2}(X,A)$, 
   defines a  crossed module.
  Then, in \cite{MR33519},
  MacLane and Whitehead showed that the spaces $X$ with $\pi_{q}(X)=0$, $q\geq 2$,
  are determined by the  crossed module
  of the pair $(X,X_{1})$, where $X_{1}$ is the 1-dimensional skeleton of $X$.
  For any cdgl $L=L_{0}\oplus L_{1}$, the geometric realization $\langle L\rangle$ is determined 
  by the  crossed module associated to the pair
  $(\langle L\rangle, \langle L_{0}\rangle)$.
  Our second main result identifies this crossed module with $\cC(L)$.

\begin{theoremv}\label{thm:whiteheadCL}
The Whitehead  crossed module associated to the simplicial pair
$(\langle L\rangle, \langle L_{0}\rangle)$
is isomorphic to the  crossed module $\cC(L)$ introduced above.
\end{theoremv}
 
 \medskip
In short, these two theorems unify the geometric realizations of complete differential graded Lie algebras
of the form $L=L_{0}\oplus L_{1}$
and of crossed modules. 
In the last section, 
we extend the correspondence between Malcev groups and pronilpotent Lie algebras to crossed modules.
 We  introduce the categories of Malcev crossed modules and  of pronilpotent Lie algebra crossed modules  
 and  prove an isomorphism of categories.
 
 \begin{theoremv}\label{thm:equivalence}
 The three following categories are isomorphic:
 \begin{enumerate}
 \item the category of pronilpotent differential graded Lie algebras of the form $L=L_{0}\oplus L_{1}$,
 \item the category of pronilpotent Lie algebra crossed modules,
 \item the category of Malcev crossed modules.
 \end{enumerate}
 Moreover, the equivalence between (1) and (3) is given by the functor $\cC$. 
 \end{theoremv}

As a next step for the future, we can consider a connected cdgl $L$ such that $H_{\geq n+1}L=0$
for some $n\geq 1$. 
 Using the ideal $J=L_{\geq n+1}\oplus dL_{n+1}$,
 the same argument used above gives a weak homotopy equivalence
 $$\langle\varphi\rangle \colon \langle L,d\rangle  
 \xrightarrow{\simeq}
  \langle L/J, \overline{d}\rangle.$$
 We conjecture that the differential $d$ defines an $n$-{cat}-group structure on $\cC(L)$ 
 (in the sense of Loday in \cite{Loday}) and that the geometric realization
  $\langle L/J, \overline{d}\rangle$ is isomorphic to the realization of this $n$-cat-group.
 
\medskip
Our program is carried out in Sections 1-7 below, whose headings are self-explanatory.

\tableofcontents
 
 \section*{Conventions and notation} 

In a graded Lie algebra $L$, the group of elements of degree~$i$ is denoted by $L_{i}$. 
A Lie algebra differential decreases the degree by 1, i.e., $dL_{i}\subset L_{i-1}$. 
If $x\in L$, we denote by $\mathrm{ad}_{x}$ the Lie derivation of $L$ defined by 
$\mathrm{ad}_{x}(y)=[x,y]$.

If there is no ambiguity, the product of two elements $m$, $m'$ of a group $M$ is denoted $mm'$. 
Sometimes, if several laws are involved, we can use some specific notation, such as $m\perp m'$ or $m*m'$,
to avoid confusion.
An action of a group $N$ on a group $M$ is always a left action and is denoted by $(n,m)\mapsto {}^nm$. 
We denote then by
$M\rtimes N$ the semi-direct product whose multiplication law is  defined by
$$(m,n)(m',n')=(m\,^{n}m',nn').$$

  \section{Background on Lie models} \label{sec:recall}
    
  A \emph{complete differential graded Lie algebra} (henceforth cdgl) 
   is a differential graded Lie algebra $L$ equipped with a decreasing filtration of differential Lie ideals, such that
   $F^1=L$, $[F^pL,F^{q}L]\subset F^{p+q}L$ and 
   $$L= \varprojlim_n L/F^n L.$$
   If no filtration is specified, it is understood that we consider the lower central series.
   
   \medskip
 Let  $V=\oplus_{i\in \mathbb Z}V_{i}$ be a rational graded vector space.
We denote by  $\L(V)$   the free graded Lie algebra on $V$, and
 by $\mathbb L^{\geq n} (V)$ the ideal of $\L(V)$ generated by the brackets of length greater than or equal to $n$.
 The \emph{completion of  $\mathbb L(V)$}
 is the inverse limit
 $$\widehat{\mathbb L}(V) = \varprojlim_n \mathbb L(V)/ \mathbb L^{\geq n} (V).$$
  This is a cdgl for the filtration given by the ideals 
  $G^n = \mbox{ker}\, (\widehat{\mathbb L}(V)\to \mathbb L(V)/\mathbb L^{>n}(V))$. 
  The correspondence $V\to \widehat{\mathbb L}(V)$ gives a left adjoint to the forgetful functor to graded rational
  vector spaces (\cite[Proposition 3.10]{book}).
  We call $\widehat{\mathbb L}(V)$ the \emph{free complete graded Lie algebra on $V$.}
  
  \medskip
  If  $\theta$ is a derivation of degree 0 on a cgl $L$, the exponential map $e^{\theta}$ is
  a cgl automorphism of $L$ defined by
  $$e^{\theta}=\sum_{i\geq 0}\frac{\theta^i}{i!}.$$
  In particular, for any $x\in L_{0}$, $e^{\ad_{x}}$ is a cgl automorphism of $L$.
  Therefore, in any cgl $L$,  the sub Lie algebra $L_{0}$ admits a group structure whose multiplication 
 law $*$ is given by the Baker-Campbell-Hausdorff product 
 (\cite[Ch.II.\S 6.Proposition 4]{MR979493},  \cite[\S 3.4]{Reut})
 and characterized by
 $$e^{\ad_{x*y}}=e^{\ad_{x}}\circ e^{\ad_{y}}.$$

 \medskip
 Now we recall  the first properties of the cosimplicial cdgl $\mathfrak{L}_{\bullet}$  (\cite[Chapter 6]{book}).
Denote as usual by $\underline{\Delta}^n$ the simplicial set in which $\underline{\Delta}^n_p$ is the set of $p+1$-uples of integers $(j_0, \dots , j_p)$ such that $0\leq j_0\leq \dots \leq j_p\leq n$. We also denote by 
$\Delta^n$ the simplicial complex formed by the non-empty subsets of $\{0,\dots,n\}$.
The subcomplex ${\dot{\Delta}}^n$ of $\Delta^n$ 
is the  simplicial complex containing the proper non-empty subsets of $\{0,\dots,n\}$.

\medskip
Finally   $s^{-1}C_*\underline{\Delta}^n$ denotes the desuspension of the simplicial chain complex 
on $\underline{\Delta}^n$ and $s^{-1}C_*{\Delta}^n$ the desuspension of the complex of simplicial chains on $\Delta^n$,
which is isomorphic to $s^{-1}N_*\underline{\Delta}^n$,
the  complex of non-degenerate chains on $\underline{\Delta}^n$.
Then, as a graded Lie algebra (without differential), we set
$$\lasu_n = \widehat{\mathbb L}(s^{-1}C_*{\Delta}^n).$$
In other words,   $\mathfrak{L}_n$ is the  free complete graded Lie algebra 
on elements 
$a_{i_0\dots i_k}$  of degree $\vert a_{i_0\dots i_k}\vert = k-1$, for all $0\leq i_0<\dots <i_k\leq n$.
For instance, we have $|a_{i}|=-1$ and $|a_{i_{0}i_{1}}|=0$.

\medskip
The family $\underline{\Delta}^{\bullet} =\{\underline{\Delta}^{n}\}_{ n\geq 0}$ is a cosimplicial object in the category of simplicial sets. 
It follows that the family $s^{-1}N_*\underline{\Delta}^{\bullet}$  is a cosimplicial object in the category of chain complexes. 
The identification $s^{-1}C_*{\Delta}^n\cong s^{-1}N_*\underline{\Delta}^n$ makes 
$s^{-1}C_*{\Delta}^n$ a cosimplicial object in the category of chain complexes.
The extension of the cofaces and codegeneracies as morphisms of Lie algebras gives morphisms of 
complete graded Lie algebras $\delta^i \colon \lasu_n\to \lasu_{n+1} $ and $\sigma^i \colon \lasu_n \to \lasu_{n-1}$.
More precisely, we have
$$\delta^i(a_{j_0 \dots  j_p}) = a_{r_0 \dots  r_p} \vspace{5mm}\mbox{ , with } r_k = 
\left\{
\begin{array}{ll}
j_k & \mbox{if } j_{k}<i,\\
j_{k}+1 & \mbox{if }j_{k}\geq i,
\end{array}\right.$$
and 
$$\sigma^i(a_{j_0 \dots  j_p}) = a_{r_0 \dots  r_p} \vspace{5mm}\mbox{ , with } r_k = \left\{\begin{array}{ll}
j_k & \mbox{if } j_{k}\leq i,\\
j_{k}-1 & \mbox{if }j_{k}> i,
\end{array}\right.$$
if $r_0<\dots <  r_p$. Otherwise,  $ \sigma^i(a_{j_0 \dots  j_p})=0$.

\begin{proposition}{\cite[Theorem 6.1]{book}}\label{def:lasu}
Each $\lasu_{n}$ can be endowed with a differential $d$ satisfying the following properties.
\begin{enumerate}[(i)]
\item The linear part $d_{1}$ of $d$ is given by
$$d_{1}a_{i_{0}\dots i_{p}}=\sum_{j=0}^p(-1)^ja_{i_{0}\dots \hat{i}_{j}\dots p}.$$
\item The generators $a_{i}$ are Maurer-Cartan elements; i.e.,
$da_{i}= -1/2 [a_{i},a_{i}]$.
\item The cofaces $\delta^i$ and the codegeneracies $\sigma^i$ are cdgl morphisms.
\item For $n\geq 2$,  
$$da_{0\dots n}= [{a_{0}},a_{0\dots n}]+ \Phi,$$
with $\Phi\in \hL(s^{-1}C_*{\dot{\Delta}}^n)$. 

\end{enumerate}
Thus, in particular, the family $\lasu_{\bullet}$ is a cosimplicial cdgl.
\end{proposition}

Let us specify  the cdgl  $\mathfrak{L}_{n}$ in low dimensions.
\begin{enumerate}
\item[$\bullet$]
 $\mathfrak{L}_0=(\mathbb L(a_0),d)$ is the free  Lie algebra  on a Maurer-Cartan element $a_0$.

\item[$\bullet$] $\mathfrak{L}_1 = (\widehat{\mathbb L}(a_{0}, a_{1},a_{01}),d)$ is the Lawrence-Sullivan interval (see \cite{LS}) with 
$$da_{01}=[a_{01},a_{1}]+\frac{\mathrm{ad}_{a_{01}}}
{e^{ \mathrm{ad}_{a_{01}}}-1}(a_{1}-a_{0}).$$

\item[$\bullet$]  $\mathfrak{L}_2 = (\widehat{\mathbb L}(a_0,a_1,a_2, a_{01}, a_{02}, a_{12}, a_{012}), d)$
is a model of the triangle (see \cite[Proposition 5.14]{book}) with the differential 
\begin{equation}\label{equa:ddeux}
d(a_{012} ) = a_{01}* a_{12}*a_{02}^{-1}-[a_{0},a_{012}].
\end{equation}
\end{enumerate}

The cosimplicial cdgl $\lasu_{\bullet}$ leads naturally to the definition of cdgl models for any simplicial set and 
to a geometric realization for any given cdgl, see \cite[Chapter 7]{book}. 
For our purpose, we only need the realization of a  cdgl $L$, defined as the simplicial set
$$\langle L\rangle =\Hom_{\cdgl}(\lasu_{\bullet},L),$$
which satisfies properties of the classical Quillen realization. For instance, for
any $n\geq 1$,  we have $\pi_{n}\langle L\rangle =H_{n-1}L$, where 
the group law of $H_{0}L$ is the BCH product (see \cite[Section 4.2]{book} or \cite[\S II.6.4]{MR979493}). 

 \section{Crossed modules and cdgl's}\label{sec:crosslie}

 For general background on crossed modules, we refer the reader to the historical papers of Whitehead
 (\cite{MR30760}, \cite{MR33519}) or to more modern presentations, such as \cite{MR1691707}, \cite{MR2841564} 
 or \cite{Loday}.
 We recall only the basics we need.
 
 \begin{definition}\label{def:crossedmod}
A crossed module $\cC=(d\colon M\to N)$ is a morphism of groups $d$ together with an action of $N$ on $M$, 
given by group automorphisms $n\mapsto (m\mapsto {}^nm)$ 
satisfying two conditions:
 \begin{enumerate}
 \item for all $m\in M$ and $n\in N$, $d({}^nm) = n  d(m)  n^{-1}$,
 \item for all $m\in M$, $m'\in M$, ${}^{d(m)}m' = m   m'  m^{-1}$.
 \end{enumerate}
 \end{definition}
 
 If  the group $N$ acts on itself by conjugation, the first property means that $d$ is compatible with the 
 $N$-action. It also implies that the group $d(M)$ is a normal subgroup of $N$ 
 and that ker$\, d$ is a sub $N$-module of $M$. 
 
On the other hand, we remark  that if $d(m)=1$,  the second property 
 implies $mm'=m'm$ which means that   $\mathrm{ker}\,d$ is included in the center of $M$. The same property   shows 
  that Im$\,d$ acts trivially on ker$\,d$ and induces thus an action of coker$\, d$ on ker$d$.   

\vspace{5mm}
 
Now  let $L= L_0\oplus L_1$ be a cdgl. 
In what follows $L_{0}$ is always considered as a group 
equipped with the BCH product denoted by $*$. 
We will prove that $d \colon L_1\to L_0$ is a crossed module. 
The first step   consists in defining a group structure on $L_1$. 
This construction was originally carried out in \cite[Definition 6.14]{book}.

\begin{proposition}\label{prop:perp}
For any   cdgl $(L,d)$ such that $L= L_0\oplus L_1$, $L_1$ admits a natural product~$\perp$ for which
the differential $d\colon (L_{1},\perp)\to (L_{0},*)$ is a group morphism.
Moreover, $a\perp b=a+b$ if $a$ and $b$ are cycles.
\end{proposition}

\begin{proof} 
The different possibilities for a definition of this law are described in \cite[Section 6.5]{book}. 
We recall here the construction for the convenience of the reader, beginning with the ``universal'' example, 
the cdgl $L'= \widehat{\mathbb L}(u_1, u_2, du_1, du_2)$, with $u_i$ in degree $1$. 
Since $HL'= 0$ there is an element $\omega$ in $L'_1$ such that 
\begin{equation}\label{equa:petitperp}
d\omega = du_1 * du_2.
\end{equation}
 Of course such an element is not unique. If $\omega'$ is another element satisfying (\ref{equa:petitperp}),  the difference $\omega-\omega'$ is a boundary 
since $H_{\geq 1}L'= 0$. This shows that the class of $\omega$ is well defined in  the cdgl quotient 
$(L'/(L'_{\geq 2}\oplus dL_2'), \overline{d})$. 
We denote this class by $u_1 \perp u_2$. By construction, it verifies
$\ov{d}(u_{1}\perp u_{2})=du_{1}*du_{2}$.

Among all the different possible choices for $\omega$,  one starts  with the  Baker-Campbell-Hausdorff series 
for $du_1 * du_2$. Replacing in each term  one and only one  $du_i$ by $u_i$ we get an element $\omega$ with $d\omega= du_{1}*du_{2}$. 
This gives,
\begin{equation}\label{FF}
\omega= u_{1}+u_{2} + \frac{1}{2}[u_{1}, du_{2}] 
+\frac{1}{12}[du_{1},[du_{1},u_{2}]]- \frac{1}{12}[du_{2},[du_{1},u_{2}]] +
\dots\end{equation}

Now, let $L$ be a cdgl with $L= L_0\oplus L_1$,   $e_1, e_2\in L_1$, and 
 $f \colon L'\to L$ the unique cdgl map sending $u_i$ to $e_i$. 
Therefore  the element $e_1\perp e_2 := f(u_1\perp u_2)$ is a well defined element in $L_1$. 
By construction, if $e_{1}$ and $e_{2}$ are cycles, using the image of the formula (\ref{FF}) in $L$, we have $e_{1}\perp e_{2}=e_{1}+e_{2}$.

For the associativity of $\perp$, we consider $L''= \widehat{\mathbb L}(u_1,u_2,u_3, du_1, du_2, du_3)$
and observe that in $L''_1/dL''_2$
 we have $(u_1\perp u_2)\perp u_3 = u_1\perp (u_2\perp u_3)$   
 because both have the same boundary. The same is thus true in $L_1$.
\end{proof}

 \vspace{3mm} With this group structure on $L_1$ we can now prove that  $L= L_0\oplus L_1$ is a crossed module.
  
 \begin{proposition}\label{prop:crossLie}
 Let $(L,d)$ be a connected complete differential graded  Lie algebra such that $L=L_0\oplus L_1$.
  Then,
 $d\colon (L_1,\perp)\to (L_0,*)$
 is a crossed module.  
 \end{proposition}

  \begin{proof}
  Recall from \cite[Definition 12.40]{book} that the group $L_{0}$ acts on $L_1$ by
 $${}^xz= e^{\mathrm{ad}_x}(z)\,, \quad \text{for all } x\in L_0, z\in L_{1}.$$
 From \cite[Corollary 4.12]{book}) it follows that,
 for any $x\in L_{0}$, $y\in L_{0}$, $z\in L_1$, we have
$$ {}^{(x*y)}z= e^{\mathrm{ad}_{x*y}}(z)= e^{\mathrm{ad}_x}(e^{\mathrm{ad}_y}z) = {}^x({}^yz).$$
To prove that the function $y\mapsto {}^xy$ is a group homomorphism, as in Proposition~\ref{prop:perp},
we consider a universal example. 
Let $E= \widehat{\mathbb L} (x,z,t, dz,dt)$ with $x$ in degree $0$, $z$ and $t$ in degree $1$, and $dx= 0$. 
Since the injection $\mathbb L(x)\to E$ is a quasi-isomorphism, we have $H_{\geq 1}(E)= 0$.
Observe  that in $E/ (E_{\geq 2}\oplus dE_2)$ we have
$$
\renewcommand{\arraystretch}{1.1}
\begin{array}{lcl}
d(^x(z\perp t))&=& e^{\mathrm{ad}_{x}}(d(z\perp t)) = 
e^{\mathrm{ad}_{x}}(dz*dt)\\
&=&
x*dz*dt*x^{-1}=x*dz*x^{-1}*x*dt*x^{-1}\\
&=&
e^{\mathrm{ad}_{x}}(dz)*e^{\mathrm{ad}_{x}}(dt)\\
&=&
d(e^{\mathrm{ad}_{x}}z)*d(e^{\mathrm{ad}_{x}}t)= d({}^xz\perp {}^xt).
\end{array}
\renewcommand{\arraystretch}{1}$$
Thus in $E_1/dE_2$, we get 
$$  {}^x(z\perp t) = {}^xz\perp  {}^x t.$$
The same is therefore also true in   $L_{1}$.

\medskip
As $x$ is a cycle, by \cite[Propositions 4.10 and 4.13]{book} we have
$$d({}^xz)= e^{\mathrm{ad}_x}(dz)= x*dz*x^{-1},$$
and Property (1) of Definition~\ref{def:crossedmod} is satisfied.
For Property (2), we use   once again  the  universal example $L'=\hL(u_1,u_2, du_1, du_2)$ 
already considered  in the proof of Proposition \ref{prop:perp}. Since in $L'_1/dL'_2$ we have
$$d({}^{du_1} u_2) = du_1*du_2*du_1^{-1}= d(u_1\perp u_2\perp u_1^{-1}),$$
 we deduce that
$${}^{du_1}u_2 = u_1\perp u_2\perp u_1^{-1},$$
and thus the same is true in $L_1$.
 \end{proof}

\vspace{3mm}\begin{remark} 
 By Proposition~\ref{prop:perp}, under the hypotheses of Proposition~\ref{prop:crossLie}, 
 we deduce that the group structures $\perp$ and $+$ coincide on $H_1L= \mathrm{ker} \,d$. 
 \end{remark}
 
 We  have thus defined a functor $\cC\colon \cdgl_{\leq 1} \to \crmod$.
 %
 \section{The crossed module of a realization and Theorem~\ref{thm:whiteheadCL}}\label{sec:crossedreal}
 
 In this section, in the case $L=L_{0}\oplus L_{1}$, 
 we establish the isomorphism between $\cC(L)$ and the Whitehead crossed module of
 $(\langle L\rangle, \langle L_{0}\rangle)$.
 
 \begin{proof}[Proof of Theorem~\ref{thm:whiteheadCL}]
The realization $\langle L\rangle =\Hom_{\cdgl}(\lasu_{\bullet},L)$ of a cdgl  $L= L_0\oplus L_1$ is a Kan complex 
(\cite[Proposition 7.13]{book}). 
We first compute $\pi_{1}(\langle L_{0}\rangle)$ and $\pi_{2}(\langle L\rangle,\langle L_{0}\rangle)$, and for that, we use  
the homotopy relation introduced in \cite[\S 3]{May}.
 
 \medskip
 Since $\lasu_{1}= (\widehat{\mathbb L}(a_{0}, a_{1},a_{01}),d)$, the map $f\mapsto f(a_{01})$ induces an isomorphism 
 of sets
 $$\langle L_0\rangle_1 = \Hom_{\cdgl} ({\lasu}_1, L_0) \xrightarrow{\cong} L_0.$$
 Since $\partial_{i}f=0$, for $i=0,\,1$, each element of $L_{0}$ defines an element of $\pi_{1}(\langle L_{0}\rangle)$.
 Now, two such 1-simplices, $g$ and $f$,
are homotopic  in $\langle L_0\rangle $ if 
  there exists a map $h \colon {\lasu}_2 \to L_0$ such that 
$\partial_{1}h=g$, $\partial_{2}h=f$ and $\partial_{0}h=0$. 
The simplex $h$ is called a homotopy from $f$ to $g$.

In the particular case $g=0$, from the simplicial structure of the realization, we get
$h(a_{02})= h(a_{12})= 0$ and $h(a_{01} ) = f(a_{01})$. 
Since $h(a_{012})= 0$, we have an equivalence: 
 $$f\sim 0 \hspace{3mm}\Leftrightarrow \hspace{3mm} 0 = dh(a_{012}) = h(a_{01}*a_{12}*a_{02}^{-1})= f(a_{01}).$$
 Therefore $\pi_1\langle L_0\rangle  = L_0$.
 
 \medskip
 To compute the relative homotopy group $\pi_2(\langle L\rangle , \langle L_0\rangle )$, we consider the set
  $$K=\left\{f\in \langle L\rangle_{2}=\Hom_{\cdgl} ({\lasu}_2, L)
\mid \partial_{i}f= 0 \text{ for } i=1,\,2 \text{ and } \partial_{0}f \in \langle L_{0}\rangle\right\}.$$
If $f\in K$, we have 
$\partial_{0}f(a_{01})=f(\delta^0(a_{01}))=f(a_{12})=f(da_{012})=df(a_{012})$ 
and thus the correspondence $K\to L_{1}$ which maps $f$ to $f(a_{012})$ is an isomorphism.
  By \cite[Definitions 3.3 and 3.6]{May}, two simplices, $f$ and $g$, of $K$ are homotopic rel $\langle L_0\rangle$ if 
 $\partial_{0}f\sim \partial_{0}g$ in $\langle L_{0}\rangle$ by a homotopy $h$,
  and there exists a 3-simplex
 $\omega\colon \lasu_{3}\to L$ such that $\partial_{0}\omega =h$, $\partial_{2}\omega=f$, $\partial_{3}\omega=g$
 and 
 $\partial_{1}\omega=0$.

 For getting an expression of these conditions at the level of cdgl's, we recall 
 (\cite[Proposition 6.16]{book}) the differential $d$
 of $\lasu_{3}$, which uses  the operation $\perp$ 
 introduced in the proof of Proposition~\ref{prop:perp}:
 \begin{equation}\label{equa:tetrahedron}
 d(a_{0123}) = e^{\ad_{a_{01}}} a_{123} - (a_{012}\perp a_{023}\perp a_{013}^{-1}).
 \end{equation}
 From $L_{\geq 2}=0$, we deduce $\omega(a_{0123})=0$.
  From the definition of $K$, we get
  $\omega(a_{123})=\partial_{0}\omega(a_{012})=h(a_{012})=0$, since $ L_{0}$ has no element of degree~1.
  We also have $\omega(a_{012})=\partial_{3} \omega(a_{012})=g(a_{012})$,
  $\omega(a_{013})=\partial_{2}\omega(a_{012})=f(a_{012})$
  and $\omega(a_{023})=\partial_{1}\omega(a_{012})=0$. 
  Thus, by applying $\omega$ to both sides of  (\ref{equa:tetrahedron}), we obtain:
$$  0=0-g(a_{012})\perp 0\perp f(a_{012})^{-1},$$
i.e., $0=g(a_{012})\perp f(a_{012})^{-1}$.
This implies $0=dg(a_{012})* df(a_{012})^{-1}$ and $df(a_{012})=dg(a_{012})$.

\smallskip
 It remains to describe $g(a_{012})\perp f(a_{012})^{-1}$. 
 From the compatibility of the differential with Lie bracket and the fact that $L_{1}$ is an abelian Lie algebra,
 we get $[g(a_{012}),dg(a_{012})]= -\frac{1}{2} d[g(a_{012}),g(a_{012})]= 0$. 
 In the BCH product  $df(a_{012})*dg(a_{012})$, all terms except the linear ones contain a bracket 
 $[df(a_{012}),dg(a_{012})]$ which becomes
 $[df(a_{012}),g(a_{012})]=[dg(a_{012}),g(a_{012})]=0$ in the formula~\eqref{FF}.
 We thus obtain
 $$g\perp f^{-1} =  g-f.$$
  We have proven $\pi_2(\langle L\rangle , \langle L_0\rangle ) \cong L_1$ and $\pi_{1}(\langle L_{0}\rangle)\cong L_{0}$. 
We also showed that the  connecting map
 $\partial\colon \pi_2(\langle L\rangle , \langle L_0\rangle )\to \pi_{1}(\langle L_{0}\rangle)$,
given by $[f]\mapsto [\partial_{0}f]$,  corresponds to $df(a_{012})$ in the previous isomorphisms.

\medskip
Consider now the action of $\pi_{1}(\langle L_{0}\rangle)=L_{0}$ on
$\pi_{2}(\langle L\rangle,\langle L_{0}\rangle)=L_{1}$.
Let $a\in L_{0}$, $b\in L_{1}$ and ${}^ab$ the element of $L_{1}$ corresponding to this action.
Recall (\cite[Lemma 4.23]{book}) that $y=e^{\ad_{a}}b$ is also an element of $L_{1}$ such that $dy=a*db*a^{-1}$.
Both constructions,
${}^ab$ and $e^{\ad_{a}}b$,
are natural, so that to prove ${}^ab=e^{\ad_{a}}b$,
we have only to prove it for the cdgl $L''$ quotient of $L'=\hL(a,u,du)$, with $\deg u=1$, by the ideal
$L'_{\geq 2}\oplus dL'_{2}$.
The required identification follows from $d({}^au)=d(e^{\ad_{a}}u)$ and the injectivity of 
$d\colon L''_{1}\to L''_{0}$.
We have thus recovered the crossed module $\cC(L)$.
 \end{proof}
 
\section{The classifying space of a crossed module}\label{sec:classcross}

By definition, the classifying space of a crossed module $\cC$ is the classifying space of the nerve of the categorical group associated to $\cC$.  Let us specify this association.
 
 \medskip 
Recall that a categorical group is a group object in the category of groups (see \cite[Section 1.1]{Loday}),
$$\xymatrix{
G\ar@<1ex>[r]^-{s}
\ar[r]_-{t}& N},$$
where $N$ is a subgroup of $G$, $s$ and $t$ are homomorphisms such that
$s\vert_N=t\vert_N=\mbox{id}_{N}$ and
$[\ker s, \ker t]=1$. 

\medskip
In \cite{Loday}, J. L.~Loday defines a categorical group associated to  a crossed module 
$\cC=(d\colon (M,\perp) \to (N,*))$
as follows:
\begin{itemize}
\item $G=M\rtimes N$ is the product $M\times N$ with the semi-direct product given by the   action of $N$ on $M$.
Thus, the product of $(m',n')$ and $(m,n)$ in $G$ is 
\begin{equation*}\label{equa:product}
(m',n')\bullet (m,n)=(m'\perp {}^{n'} m, n' *n).
\end{equation*}
\item An element $(m,n)$ of $G$ has for source and target, respectively,
\begin{equation*}\label{equa:sandt}
s(m,n)=dm\ast n\quad \text{and} \quad t(m,n)=n.
\end{equation*}
\end{itemize}

\vspace{2mm} Thus, the group $N$ is interpreted as the group of objects viewed in $G$ as $\{1\}\times N$.
The group $G=M\rtimes N$  is the group of arrows with the
morphisms $s$ and $t$ giving the source and the target. 
Two elements $(m',n')$ and $(m,n)$ are composable if 
$$n'=t(m',n')=s(m,n)=dm*n.$$
In this case the composition is defined by
\begin{equation*}\label{equa:compo}
(m',n')\circ (m,n) = (m'\perp m,n).
\end{equation*}

We deduce easily from Property (1)  of Definition~\ref{def:crossedmod} that $s$ and $t$ are group homomorphisms.
We  also verify that the source of a composite is the source of the first factor and the target is the target of the second factor:
$$
\begin{array}{lcl}
s(m'\perp m,n)&=&
d(m'\perp m)*n=dm'*dm*n=dm'*n'\\
&=& s(m',n'),\\
t(m'\perp m,n)&=&n=t(m,n).
\end{array}
$$
Finally,  composition is a group homomorphism,  see \cite[Lemma 2.2]{Loday}.

\vspace{5mm} The usual nerve of a category is a simplicial set. When the category is a categorical group, we obtain naturally a simplicial group.  
Let us describe the nerve of the categorical group 
associated to a crossed module $\cC=(d\colon (M,\perp) \to (N,*))$. 
We have
$$\xymatrix{
\ner_{1}=M\rtimes N\ar@<1ex>[r]^-{d_{1}}
\ar[r]_-{d_{0}}& \ner_{0}=N},$$
with $d_{0}(m,n)=t(m,n)=n$, $d_{1}(m,n)=s(m,n)=dm*n$
and  $s_{0}\colon \ner_{0}\to \ner_{1}$ is the canonical injection $N\to M\rtimes N$.

\medskip
An element of $\ner_{k}$ is a sequence $(m_{i},n_{i})_{1\leq i\leq k}$ such that
$$n_{i}=t(m_{i},n_{i})=s(m_{i-1},n_{i-1})=dm_{i-1}*n_{i-1}.$$
 As  the $n_{i}$, for $i\geq 2$,  are determined by $n_1$ and the family
$(m_{i})_{1\leq i\leq k}$, the sequence $(m_i,n_i)_{i\leq k} $ can be identified with the sequence 
$$(m_k, m_{k-1}, \dots , m_1, n_1)\in M^k\times N.$$ 
In particular,  
\begin{equation}\label{equa:nerve}
\ner_{k}=M^k\times N.
\end{equation}
Each $\ner_k$ is a group,  
the multiplication being given  component wise. With the identification (\ref{equa:nerve}),
 this product is given by 
\begin{equation*}\label{equa:productnerve}
((m_{i})_{1\leq i\leq k},n)\bullet ((m'_{i})_{1\leq i\leq k},n')=
((m_{i}\perp ^{d(\perp_{j=1}^{i-1} m_{j})*n}m'_{i})_{1\leq i\leq k}, n*n').
\end{equation*}
The boundary and degeneracy maps of $\ner_*$ are morphisms of groups defined as usual by:
$$\renewcommand{\arraystretch}{1.1}
\left\{
\begin{array}{ll}
d_0(m_k, \dots , m_1,n) = (m_k, \dots, m_2, d(m_1)*n),&\\
d_i(m_k, \dots , m_1,n) = (m_k, \dots , m_{i+1}\perp m_{i}, \dots , m_1,n),
&0<i<k, \\
 d_k(m_k, \dots , m_1, n) = (m_{k-1}, \dots , m_1, n),&\\
s_{i}(m_k, \dots , m_1, n)=(m_{k},\dots,m_{i},1,m_{i-1},\dots,m_{1},n),
&0\leq i\leq k.
\end{array}\right.
\renewcommand{\arraystretch}{1}$$
The identity $e_k\in \ner_k$ is the element $(1, \dots ,1,1)$.

\medskip
Recall from \cite[Definition 3.20]{MR279808} or \cite[Page 255]{MR1711612} 
the classifying functor $\ov{W}$ which goes 
from the category of simplicial groups to the category of reduced simplicial sets.
The classifying space $B\cC$ of the crossed module $\cC$ is the space obtained by composing 
$\ner_*$ with  $\ov{W}$:
$$B\cC=\ov{W}(\ner_*).$$

By definition of $\ov{W}$, we have
$$
(B\cC)_{k}
=
\{(h_{k-1},\dots,h_{0})\mid h_{i}\in \ner_{i}\}.
$$
The boundaries and degeneracies are given by
\begin{equation*}
\renewcommand{\arraystretch}{1.1} %
\left\{\begin{array}{rcl}
d_{0}(h_{k-1},\dots,h_{0})
&=&
(h_{k-2},\dots,h_{0}),\\
d_{i}(h_{k-1},\dots,h_{0})
&=&
(d_{i-1}h_{k-1},\dots, d_{0}h_{k-i}\bullet h_{k-i-1},h_{k-i-2},\dots,h_{0}),
\;0<i<k,\\
d_{k}(h_{k-1},\dots,h_{0})
&=&
(d_{k-1}h_{k-1},\dots,d_{1}h_{1}),\\
s_{0}(h_{k-1},\dots,h_{0})
&=&
(1,h_{k-1},\dots,h_{0}),\\
s_{i}(h_{k-1},\dots,h_{0})
&=&
(s_{i-1}h_{k-1},\dots,s_{0}h_{k-i},1,h_{k-i-1},\dots,h_{0}),
\;0< i\leq k.
\end{array}\right.
\renewcommand{\arraystretch}{1}
\end{equation*}

 In particular, in low dimensions, we have
$B\cC_0= 1$,
$B\cC_1= N$,
$B\cC_2 = (M\rtimes N)\times N$,
$B\cC_3 = (M^2\rtimes N)\times (M\rtimes N)\times N$.

\section{The classifying space functor $\ov{W}$ and twisting functions}

Let $A_*$ be a simplicial set. By \cite[Corollary 27.2]{May}, there is a bijective correspondence between morphisms of simplicial sets $\varphi \colon A_* \to \ov{W}\circ \ner_*=B\cC$ and twisting functions 
$$ \tau=\{\tau_{k}\colon A_k\to \ner_{k-1}\}_{k\geq 1}.$$
Recall that (\cite[Definition 18.3]{May}) a twisting function $\tau$ is a family of maps $\tau_k \colon A_k\to \ner_{k-1}$ satisfying
 the following properties for $x\in A_k$:
$$\renewcommand{\arraystretch}{1.1}
\left\{\begin{array}{rcl}
d_{0}\tau x
&=&
\tau d_{1}x\bullet (\tau d_{0}x)^{-1},\\
d_{i}\tau x
&=&
\tau d_{i+1}x, \;i>0,\\
s_{i}\tau x
&=&
\tau s_{i+1}x,\; i\geq 0,\\
\tau s_{0}x
&=&
e_k\in \ner_k.
\end{array}\right.
\renewcommand{\arraystretch}{1}$$

\medskip
The simplicial map $\varphi_{k}\colon A_k \to (B\cC)_{k}$
associated to the twisting function $\tau$  
is given by
\begin{eqnarray}
\label{F1}\varphi_{k} x=(\tau x, \tau d_{0}x,\dots, \tau d_{0}^{k-1}x).
\end{eqnarray}

Conversely (\cite[Page 88]{May}) the twisting function $\tau$ associated to a simplicial morphism $\varphi \colon A_*\to \ov{W}(\ner_*)$ is defined by
$$\tau = \tau(\ner_*) \circ \varphi,$$
where $\tau(\ner_*)$ is the twisting function associated to the identity on $\ov{W}(\ner_*)$,
$$\tau(\ner_*) \colon \ov{W}(\ner_*)_k \to \ner_{k-1}\,, $$
defined by $$ \tau(\ner_*)(g_{n-1}, \dots , g_0) = g_{n-1}.$$

\section{Proof of  Theorem~\ref{thm:realizationandcrossed}}\label{sec:proofthm}

First we compute the simplicial set 
$\langle L\rangle_{\bullet} = \Hom_{\cdgl}({\mathfrak L}_{\bullet}, L)$ in the case $L= L_0\oplus L_1$.
By $L_{\geq 2}=0$ and \cite[Corollary 6.5]{book}, we have  isomorphisms
$$\Hom_{\cdgl}(\lasu_{k},L)\cong\Hom_{\cdgl}( (\widehat{\mathbb L}((s^{-1}\Delta^k)_{\leq 2}),d),L)
\cong
\Hom_{\cdgl}( (\widehat{\mathbb L}((s^{-1}\Lambda_{0}^k)_{\leq 2}),d),L).
$$
Since any morphism of codomain $L$ vanishes on elements of negative degree, we can quotient 
 by the differential ideal generated by the generators of degree~-1.
 This gives as free cgl
 $$
 \overline{\lasu}_k = (\widehat{\mathbb L}(a_{ij}, a_{0st}),d) \hspace{2mm}\mbox{with }0\leq i<j\leq k\hspace{2mm}\mbox{and } 0<s<t\leq k.
 $$
Finally, in view of the differential in $\lasu_{2}$, recalled in \eqref{equa:ddeux}, 
the differential of $\ov{\lasu}_{k}$ satisfies
$$da_{ij}=0\text{ and }
da_{0st}=a_{0s}*a_{st}*a_{0t}^{-1}.$$
In the rest of this text, we will use that for all k, there exists an isomorphism
$$\langle L\rangle_{k}=\Hom_{\cdgl}(\lasu_{k},L)=\Hom_{\cdgl}(\ov{\lasu}_{k},L).$$

\begin{proposition}
\label{XY}
If $L= L_0\oplus L_1$, then the morphism
$$\Psi \colon \Hom_{\cdgl}(\lasu_k, L) \to L_0^k\times L_1^{\frac{k(k-1)}{2}}$$
given by $\Psi (f) = ((f(a_{r\, r+1}))_{0\leq r< k}, (f(a_{r,r+1,s}))_{r+1<s\leq k})$
is an isomorphism.
\end{proposition}

\begin{proof} For the sake of simplicity write for $i<j$, $a_{ji}= a_{ij}^{-1}$, and for $0\leq i<j<r\leq k$,
$$\renewcommand{\arraystretch}{1.2}
\begin{array}{ll}
a_{irj}= a_{ijr}^{-1},&\\ 
a_{rij}= {}^{a_{ri}}a_{ijr} ={}^{a_{ir}^{-1}}a_{ijr}\mbox{},
&a_{jir}={}^{a_{ji}} a_{irj}={}^{a_{ij}^{-1}}a_{ijr}^{-1},\\ 
a_{jri}= {}^{a_{ji}} a_{ijr} ={}^{a_{ij}^{-1}}a_{ijr},\hspace{5mm}\mbox{}
&a_{rji}= {}^{a_{ri}} a_{irj}={}^{a_{ir}^{-1}}a_{ijr}^{-1}.
\end{array}
\renewcommand{\arraystretch}{1}$$
With this notation, when the integers $i,j,r$ are all different from each other and between $0$ and $k$, we have
$$da_{ijr}= a_{ij}*a_{jr}*a_{ri}.$$
Suppose that   the elements $f(a_{r,r+1})$ and  $f(a_{r,r+1,t})$, with $r+1<t$,  are defined.
Then the other elements, $f(a_{r, r+s})$ and $f(a_{r, r+s, t})$ with $r+s<t$, can be derived by induction on $s$ 
from the  formulas
$$f(a_{r, r+s+1})= df(a_{r,r+1, r+s+1})^{-1} * f(a_{r,r+1})* f(a_{r+1, r+s+1})$$
and
$$f(a_{r, r+s+1,t}) = {}^{f(a_{r,r+1})}
\left( f(a_{r+1, r, r+s+1}) \perp f(a_{r+1, r+s+1, t})\perp f(a_{r+1, t, r})\, \right).$$
This shows that $\Psi$ is injective. The same construction process shows that $\Psi$ is also surjective.
\end{proof}

\vspace{5mm}
The   isomorphism of our main theorem is based on a family $\tau$ of  maps\\
$$ \tau_k \colon \Hom_{\cdgl}(\lasu_k, L) \to \ner_{k-1}, k\geq 1,$$
defined by 
$$\tau_{k}f=(m_{k-1},\dots,m_{1},n)\in M^{k-1}\times N,$$
with $n=f(a_{01})$, $m_{1}=(f(a_{012}))^{-1}$ and $m_{i}=(f(a_{01(i+1)}))^{-1}\perp f(a_{01i})$, for $i\geq 2$.
In  low dimensions, this gives:
$$\renewcommand{\arraystretch}{1.1}
\left\{
\begin{array}{l}
 \tau_{1} f=f(a_{01})\in N ,\\
 \tau_{2} f=(f(a_{012})^{-1}, f(a_{01}))\in M\times N,\\
 \tau_{3} f= (f(a_{013})^{-1}\perp f(a_{012}),f(a_{012})^{-1}, f(a_{01}))\in M^2\times N.
 \end{array}\right.
\renewcommand{\arraystretch}{1}
$$

\begin{proposition} The family $\tau$ is a twisting function.\end{proposition}

\begin{proof}
Observe that $m_{i+1}\perp m_{i}=f(a_{01(i+2)})^{-1} \perp f(a_{01i})$. Thus, the index $i+1$ disappears in the
expression of $d_{i}\tau_{k}f$ and we get $d_{i}\tau_{k}f=\tau_{k-1}d_{i+1}f$ for $0<i<k-1$.
A similar argument gives also the result for $d_{k-1}$.
We have reduced the problem to proving the more subtle equality 
involving $d_{0}$.
We use an induction, supposing the result is true for $\tau_{j}$, $j<k$, and considering $\tau_{k}$.
Due to the inductive step, we can concentrate the computations on the left hand factor. 
From the definitions, we have
$$\begin{array}{lcl}
\tau_{k-1} d_{1}f
&=&
(f(a_{02k}))^{-1}\perp f(a_{02(k-1)}),\dots,f(a_{02})),\\
\tau_{k-1} d_{0}f
&=&
((f(a_{12k})^{-1}\perp f(a_{12(k-1)}),\dots,f(a_{12})),\\
d_{0}\tau_{k} f
&=&
((f(a_{01k})^{-1}\perp f(a_{01(k-1)}),\dots,(df(a_{012}))^{-1}*f(a_{01}) ).
\end{array}$$
We  determine the product of the two last terms,
$$
d_{0}\tau_{k}f\bullet \tau_{k-1}d_{0}f=
(f(a_{01k})^{-1}\perp f(a_{01(k-1)})
\perp\;
 ^{\gamma} (f(a_{12k})^{-1}\perp f(a_{12(k-1)})),\dots),
$$
where $\gamma=dm_{k-2}\ast dm_{k-1}\ast\dots \ast dm_{1}\ast n= f(a_{0(k-1)})*(f(a_{1(k-1)}))^{-1}$.
To obtain the equality with $\tau_{k-1}d_{1}f$, we consider the following computation  in $\overline{\lasu}_k$:
\begin{eqnarray*}
d(a_{01k}^{-1}\perp a_{01(k-1)}
\perp\;
 ^{a_{0(k-1)}*a_{1(k-1)}^{-1}} (a_{12k}^{-1}\perp a_{12(k-1}))
 &=&
 a_{0k}*a_{2k}^{-1}
*a_{2(k-1)}*a_{0(k-1)}^{-1}\\
&=&
d(a_{02k}^{-1}\perp a_{02(k-1)}).
\end{eqnarray*}
Similar computations give the corresponding equalities for degeneracy maps.
\end{proof}

\medskip
Denote by $\varphi$ the  morphism of simplicial sets induced by the previous twisting function~$\tau$,
$$\varphi \colon \Hom_{\cdgl}({\mathfrak L}, L) \to B\cC(L).$$
The following result finishes the proof of the Theorem.

\begin{proposition} 
The morphism $\varphi$ is an isomorphism of simplicial sets.
\end{proposition}

\begin{proof} Recall from (\ref{F1}) that 
$$\varphi_kf = (\tau f, \tau d_0f, \dots , \tau d_0^{k-1}f).$$ 
On the other hand, using $d_{0}f=f\delta^0$, we get
$\tau d_0f = (m'_{k-2}, \dots , m'_1, n')$, with $n'=f(a_{12})$, $m'_1= f(a_{123})^{-1}$ and for 
$i>1$, $m'_i = f(a_{12(i+2)})^{-1}\perp f(a_{12(i+1)})$. 
By iteration from $(d_{0})^\ell f=f(\delta^0)^\ell$, we deduce that
the image of $\varphi_k$ is the linear subspace generated by the elements 
$f(a_{r,r+1})$, for $0\leq r<k$, and  $f(a_{r, r+1,s})$, for $r+1<s\leq k$. 
The result follows thus from Proposition \ref{XY}.
\end{proof}
\section{Malcev crossed modules and Theorem~\ref{thm:equivalence}}\label{sec:Malcev}

In this section, we establish an isomorphism of categories between $\cdgl_{\leq 1}$
and a subcategory of crossed modules.
We use the Lie algebra crossed modules introduced by Kassel and Loday in \cite{MR694130}.
We begin with a reminder of \cite{MR694130}.

\medskip
In Definition~\ref{def:crossedmod}, the group action of $N$ on $M$
 corresponds to a homomorphism from $N$ in the group of automorphisms of $M$.
For Lie algebras, $\gn$ and $\gm$, an action of $\gn$ on $\gm$ corresponds 
to a Lie morphism  $\gv\colon\gn\to\Der(\gm)$  in the Lie algebra of derivations of $\gm$. 
The action of $n\in\gn$ on $m\in\gm$ is denoted $\gv(n).m$.
We can now state \cite[D\'efinition A.1]{MR694130}.

\begin{definition}\label{def:crossedlie}
A Lie algebra crossed module  is a morphism of Lie algebras, $\gu\colon \gm\to \gn$,
together with an action  $\gv\colon \gn\to \Der(\gm)$, satisfying two conditions:
 \begin{enumerate}
 \item for all $m\in \gm$ and $n\in \gn$, $\gu(\gv(n).m)=[n,\gu(m)]$,
 \item for all $m\in \gm$, $m'\in \gm$, $\gv(\gu(m)).m'=[m,m']$.
 \end{enumerate}
\end{definition}

We now introduce the ``rational'' versions of crossed modules.
If $G$ is a group, $G^k=[G,G^{k-1}]$ denotes the lower central series of $G$. 

\begin{definition}\label{def:malcev}\mbox{}
\begin{enumerate}
\item A group $G$ is a \emph{Malcev group} (or prounipotent rational group) if each $G^k/G^{k+1}$ is a $\Q$-vector space, $\dim G/G^2<\infty$ and
$G=\varprojlim_{{k}} G/G^k$.
\item A crossed module $d\colon M\to N$ is a \emph{Malcev crossed module} if $M$ and $N$ are Malcev groups
and the action of $N$ on $M$ satisfies
$({}^nm)m^{-1}\in M^{k+1}$ for all $m\in M^k$, $n\in N$.
\end{enumerate}
\end{definition}

If $\gm$ is a Lie algebra, $\gm^k=[\gm,\gm^{k-1}]$ denotes the lower central series of $\gm$.

\begin{definition}\label{def:pronilLie}\mbox{}
\begin{enumerate}
\item
A Lie algebra $\gm$ is \emph{pronilpotent} if $\dim \gm/\gm^2<\infty$ and $\gm=\varprojlim_{k}\gm/\gm^k$.
\item A Lie algebra crossed module $\gu\colon \gm\to \gn$ is \emph{pronilpotent} if $\gm$ and
 $\gn$ are pronilpotent Lie algebras
and  the action
$\gv\colon \gn\to \Der(\gm)$ satisfies
$\gv(\gn).\gm^k\subset \gm^{k+1}$.
\end{enumerate}
\end{definition}

\begin{remark}
The completion of a Lie algebra $\gm$ satisfying $\dim \gm/\gm^2<\infty$  is the Lie algebra 
$\widehat{\gm} = \lim_k \gm/\gm^k$. This is a pronilpotent Lie algebra since
$\widehat{\gm}= \lim_k \widehat{\gm}/\widehat{\gm}^k$. 
\end{remark}

If a Lie algebra $\gm$ acts on a vector space $V$, we denote by $V^k$ the sequence of subspaces
$V^0=V,\;V^k=\gm.V^{k-1}$.

\begin{definition}\label{def:pronildgc}
The action of $\gm$ on $V$ is \emph{pronilpotent} if $V=\varprojlim_{k}V^k$.
In particular, a cdgl  $L=L_{0}\oplus L_{1}$ is \emph{pronilpotent} if the Lie algebra $L_{0}$ is pronilpotent 
and if the adjoint action of $L_{0}$ on $L_{1}$ is pronilpotent.
\end{definition}

\begin{proof}[Proof of Theorem~\ref{thm:equivalence}]
We only define the correspondences for objects, the extension to morphisms being immediate.

(1) $\Rightarrow$ (2).
We start with a pronilpotent cdgl $L=L_{0}\oplus L_{1}$ and we construct a pronilpotent Lie algebra crossed module
$\gu\colon \gm\to \gn$ with action $\gv\colon \gn\to \Der(\gm)$. We denote $d$ the differential of $L$ and $[-,-]$ its
bracket.

We set $\gn=L_{0}$, $\gm=L_{1}$. The bracket on $\gn$ is the bracket  of $L_{0}$ and the bracket on
$\gm$ is defined by 
$$[a,b]'=[da,b], \text{ for } a,b\in L_{1}.$$
We check that $[-,-]'$ is an (ungraded) Lie bracket. Since  $[a,b]=0$, the antisymmetry follows from
$$0=d[a,b]=[da,b]+[db,a]=[a,b]'+[b,a]'.$$
The proof is similar for the Jacobi identity.
The morphism $\gu\colon \gm\to \gn$ is the differential~$d$; this is a Lie algebra morphism:
$$\gu([a,b]')=d[da,b]=[da,db]=[\gu(a),\gu(b)],\;\text{for all } a,\,b\in \gm.$$
The action $\gv\colon \gn\to \Der(\gm)$ is given by the adjoint action, $\gv(x)=\ad_{x}$. 
The formulae (1) and (2) of Definition~\ref{def:crossedlie} also follow immediately: 
let $a,\,b\in \gm=L_{1}$ and $x\in \gn=L_{0}$, 
we have
\begin{eqnarray*}
\gu(\gv(x).a)&=& d(\ad_{x}(a))=d[x,a]=[x,da]=[x,\gu(a)],\\
\gv(\gu(a)).b&=&\ad_{da}(b)=[da,b]=[a,b]'.
\end{eqnarray*}
By definition, since $L$ is pronilpotent the associated Lie algebra crossed module is also pronilpotent.

\medskip
(2) $\Rightarrow$ (1). 
We start with a pronilpotent Lie algebra crossed module
$\gu\colon \gm\to \gn$ with action $\gv\colon \gn\to \Der(\gm)$ and we construct a pronilpotent cdgl $L=L_{0}\oplus L_{1}$.
We define $L_{0}=\gn$ as Lie algebra and $L_{1}=\gm$ as vector space.
For $a\in L_{1}$ and $x\in L_{0}$, we set $[x,a]=\gv(x).a$ and $d=\gu$.
We check easily that $d$ is a derivation and $L=L_{0}\oplus L_{1}$ is pronilpotent.

The associations (1) $\Rightarrow$ (2) and (2) $\Rightarrow$ (1) give the desired isomorphism
 of categories for the two first points of the statement.

\medskip
(2) $\Rightarrow$ (3). 
We start with a pronilpotent Lie algebra crossed module
$\gu\colon \gm\to \gn$ with action $\gv\colon \gn\to \Der(\gm)$ and we construct a Malcev crossed module
$d\colon M\to N$. We define $M$ and $N$ to be the vector spaces $\gm$ and $\gn$ respectively, with the
group structure given by the Baker-Campbell-Hausdorff product, and set $d=\gu$. 
The action $\gv$ extends in an action by 
$e^{\gv}$: for $n\in N=\gn$, $m\in M=\gm$, we set
$${}^nm=e^{\gv(n)}(m).$$
As $\gv$ is a morphism of Lie algebras, we have $\gv[n,n']=[\gv(n),\gv(n')]$, for all $n,\,n'\in N$, and so,
 the Baker-Campbell-Hausdorff formula implies $\gv(n*n')=\gv(n)*\gv(n')$ and
 $${}^{(n*n')}m=e^{\gv(n*n')}(m)=e^{\gv(n)}(e^{\gv(n')}(m)).
$$
Thus, we have a group action. The two additional properties of Malcev crossed modules are easily deduced from
the corresponding properties of Lie algebra crossed modules as well as the pronilpotency conditions.

\medskip
(3) $\Rightarrow$ (2). 
As we do for the cases (1) and (2), the previous process can be reversed. 
We associate a pronilpotent Lie algebra to a Malcev group, replacing the exponential by the functor
$L\mapsto \log(1+L)$. 
The only significative point is the 
construction of the Lie algebra action $\gv\colon \gn\to \Der(\gm)$ from the group action $\nu\colon N\to \Aut(M)$;
this is done by
$$\gv(n).m=\log(1+\nu(n))(m).$$

\medskip
We end with the study of the composition (1) $\Rightarrow$ (2) $\Rightarrow$ (3). We start with $L=L_{0}\oplus L_{1}$
and in the step (2), we define a bracket on $L_{1}$ by $[a,b]'=[da,b]$. Then, in the second implication, we endow
$L_{1}$ with a group law coming from the Baker-Campbell-Hausdorff formula, $a*b=\log(e^ae^b)$. This formula
can be written
\begin{eqnarray*}
a*b
&=&
a+b+\frac{1}{2}[a,b]'+\frac{1}{12}[a,[a,b]']'- \frac{1}{12}b,[a,b]']' +\dots\\
&=&
a+b+\frac{1}{2}[da,b]+\frac{1}{12}[da,[da,b]]- \frac{1}{12}db,[da,b]] +\dots.
\end{eqnarray*}
This is exactly the expression of $a\perp b$ given in the formula (\ref{FF}).
We recover the group law on $L_{1}$ in $\cC(L)$. The rest of the verification is straightforward. 
\end{proof}

\providecommand{\bysame}{\leavevmode\hbox to3em{\hrulefill}\thinspace}
\providecommand{\MR}{\relax\ifhmode\unskip\space\fi MR }
\providecommand{\MRhref}[2]{%
  \href{http://www.ams.org/mathscinet-getitem?mr=#1}{#2}
}
\providecommand{\href}[2]{#2}

\end{document}